\documentclass[b5paper,twoside,10pt]{article}

\makeatletter

\newif\if@final
\newif\if@finalweb
\newif\if@proofs
\newif\if@proofsline
\newif\if@oldstyle
\newif\if@imfm
\@finalfalse
\@finalwebfalse
\@proofsfalse
\@imfmfalse

\if@proofsline\RequirePackage[displaymath, mathlines]{lineno}\fi
\RequirePackage{amsmath}
\RequirePackage{amsthm}
\RequirePackage{amsfonts}
\RequirePackage{geometry}
\if@final\RequirePackage{graphicx}\RequirePackage{lastpage}\RequirePackage{ccicons}\fi
\if@proofs\RequirePackage{eso-pic}\fi
\if@finalweb
\if@proofs
\RequirePackage[hyperfootnotes=false,colorlinks=true,pagebackref=true]{hyperref}
\else
\RequirePackage[hyperfootnotes=false,colorlinks=true]{hyperref}
\fi
\else
\RequirePackage[hyperfootnotes=false,colorlinks=true]{hyperref}
\fi

\numberwithin{equation}{section}

\RequirePackage{times}

\if@imfm
\def\@urladdress{http://amc.imfm.si}
\else
\def\@urladdress{http://amc-journal.eu}
\fi
\newcount\@authorcount
\newcount\@footcount
\if@oldstyle
\else
\fi
\if@final\else\fi
\def\@emails{}
\def\@authors{}
\def\p@pertitle{}
\def\blfootnote{\xdef\@thefnmark{}\@footnotetext} 
\long\def\symbolfootnote[#1]#2{\begingroup%
\def\thefootnote{\fnsymbol{footnote}}\footnote[#1]{#2}\endgroup} 

\newcommand{\journalnumber}[1]{\def\@journalnumber{#1}}
\newcommand{\journalyear}[1]{\def\@journalyear{#1}}
\newcommand{\journalpages}[2]{%
}
\newcommand{\dates}[3]{%
\def\@receivedate{#1}%
\def\@acceptdate{#2}%
\def\@onlinedate{#3}}

\newenvironment{frontmatter}
{\thispagestyle{amctitle}%
\if@final
  \begin{minipage}{0.3\linewidth}
  \raggedright
  \end{minipage} \hfill
  \begin{minipage}{0.3\linewidth}
  \centering 
  ~
  \end{minipage} \hfill
  \begin{minipage}{0.3\linewidth}
  \centering 
  \if@finalweb\includegraphics[width=\textwidth]{AMCc.pdf}\else\includegraphics[width=\textwidth]{AMC.pdf}\fi
  \end{minipage}
  \begin{center}
  \if@oldstyle
  \else
  {\scriptsize \@issn}\\
  \fi
  ~
  \end{center}
\else
  \begin{center}
  \end{center}
\fi
\vskip 10pt
\if@proofsline\global\linenumbers\fi%
}
{\vskip 20pt%
\blfootnote{\raggedright \ifnum\@authorcount=1\textit{E-mail address:}\else\textit{E-mail addresses:}\fi ~\@emails}%
\if@proofsline\global\linenumbers\fi%
}

\newcommand{\footadvance}[1]{%
\global\advance\@footcount by #1
}

\newcommand{\titledata}[2]{%
\global\def\p@pertitle{#1}%
\vskip 5mm
\begin{center}\LARGE\bfseries%
#1\if!#2!\else
\global\advance\@footcount by 1%
\symbolfootnote[\@footcount]{#2}\fi\end{center}%
\vskip 1mm
}

\newcommand{\authordata}[4]{
\advance\@authorcount by 1%
\xdef\@emails{\@emails\ifnum\@authorcount>1 , \fi {#3} (#1)}%
\xdef\@authors{\@authors\ifnum\@authorcount>1 , \fi {#1}}
\begin{center}%
\if!#4!
{\large #1}\\[1mm]
{\itshape #2}
\else
\global\advance\@footcount by 1%
{\large #1} \symbolfootnote[\@footcount]{#4}\\[1mm]
{\itshape #2}%
\fi
\end{center}}

\newcommand{\authordatanoemail}[3]{
\begin{center}%
\if!#3!
{\large #1}\\[1mm]
{\itshape #2}
\else
\global\advance\@footcount by 1%
{\large #1} \symbolfootnote[\@footcount]{#3}\\[1mm]
{\itshape #2}%
\fi
\end{center}}

\newcommand{\authordataff}[5]{
\advance\@authorcount by 1%
\xdef\@emails{\@emails\ifnum\@authorcount>1 , \fi #3 (#1)}%
\xdef\@authors{\@authors\ifnum\@authorcount>1 , \fi {#1}}
\begin{center}%
\global\advance\@footcount by 1%
{\large #1} \symbolfootnote[\@footcount]{#4}
\global\advance\@footcount by 1%
\symbolfootnote[\@footcount]{#5}
\\[1mm]
{\itshape #2}%
\end{center}}

\newcommand{\authordatawithfootnote}[4]{
\advance\@authorcount by 1%
\xdef\@emails{\@emails\ifnum\@authorcount>1 , \fi #3 (#1)}%
\xdef\@authors{\@authors\ifnum\@authorcount>1 , \fi {#1}}
\begin{center}%
{\large #1} $^{#4}$\\[1mm]
{\itshape #2}%
\end{center}}

\newcommand{\authordatanoaffil}[3]{
\advance\@authorcount by 1%
\xdef\@emails{\@emails\ifnum\@authorcount>1 , \fi #2 (#1)}%
\xdef\@authors{\@authors\ifnum\@authorcount>1 , \fi {#1}}
\if!#3!
{\large #1}
\else
\global\advance\@footcount by 1%
{\large #1} \symbolfootnote[\@footcount]{#3}%
\fi}

\newcommand{\keywords}[1]{\def\@keywords{#1}}
\newcommand{\msc}[1]{\def\@msc{#1}}

\renewenvironment{abstract}
{
\hrule height 0.25pt
\vskip 5pt
\noindent \textbf{Abstract}
\vskip 5pt
}
{
\vskip 5pt
\noindent \textit{\small Keywords:~\@keywords}

\vskip 3pt
\noindent \textit{\small Math.\ Subj.\ Class.: \@msc}
\vskip 5pt
\hrule height 0.25pt
} 

\renewcommand{\section}{%
\@startsection{section}{1}{0pt}{-\baselineskip}{0.5\baselineskip}{\large\bfseries}%
}

\renewcommand{\subsection}{%
\@startsection{subsection}{2}{0pt}{-\baselineskip}{0.5\baselineskip}{\bfseries}%
}

\theoremstyle{plain}
\newtheorem{thm}{Theorem}[section]

\newtheorem{lem}[thm]{Lemma}

\theoremstyle{definition}

\newtheorem{rem}[thm]{Remark}
\newtheorem{remark}[thm]{Remark}

\if@proofs
  \newcommand{\ps@amc}{%
    \renewcommand\@evenfoot{}
    \renewcommand\@oddfoot{}
  }
  \newcommand{\ps@amctitle}{%
    \renewcommand\@oddhead{}
    \let\@evenhead\@oddhead
    \let\@oddfoot\@evenfoot
  }
\else
  \newcommand{\ps@amc}{%
    \renewcommand\@evenfoot{}
    \renewcommand\@oddfoot{}
  }
  \newcommand{\ps@amctitle}{%
    \renewcommand\@oddhead{}
    \let\@evenhead\@oddhead
    \let\@oddfoot\@evenfoot
  }
\fi

\def\enumerate{%
  \ifnum \@enumdepth >\thr@@\@toodeep\else
    \advance\@enumdepth\@ne
    \edef\@enumctr{enum\romannumeral\the\@enumdepth}%
      \expandafter
      \list
        \csname label\@enumctr\endcsname
        {\usecounter\@enumctr\def\makelabel##1{\hss\llap{##1}}%
          \itemsep=0pt}%
  \fi}
\def\itemize{%
  \ifnum \@itemdepth >\thr@@\@toodeep\else
    \advance\@itemdepth\@ne
    \edef\@itemitem{labelitem\romannumeral\the\@itemdepth}%
    \expandafter
    \list
      \csname\@itemitem\endcsname
      {\def\makelabel##1{\hss\llap{##1}}%
        \itemsep=0pt}%
  \fi}

\renewenvironment{thebibliography}[1]
     {\section*{\refname}%
      \list{\@biblabel{\@arabic\c@enumiv}}%
           {\settowidth\labelwidth{\@biblabel{#1}}%
            \leftmargin\labelwidth
            \advance\leftmargin\labelsep
            \@openbib@code
            \usecounter{enumiv}%
            \let\p@enumiv\@empty
            \renewcommand\theenumiv{\@arabic\c@enumiv}
            \itemsep=0.0pt}%
      \sloppy
      \small
      \clubpenalty4000
      \@clubpenalty \clubpenalty
      \widowpenalty4000%
      \sfcode`\.\@m}
     {\def\@noitemerr
       {\@latex@warning{Empty `thebibliography' environment}}%
      \endlist}
\let\@openbib@code\@empty

\if@proofs
\newcommand{\BackgroundPic}{
\put(0,-50){
\parbox[b][\paperheight]{\paperwidth}{%
\vfill
\centering
\includegraphics[width=\paperwidth,height=\paperheight,keepaspectratio]{proofs.pdf}%
\vfill
}}}
\fi

\newcommand{\amcend}
{
\if@final
\newwrite\contentsfile
\immediate\openout\contentsfile=contentsfile.txt
\makeatletter
\immediate\write\contentsfile{\@backslashchar contentsline\@charlb subsection\@charrb
\@charlb\@backslashchar textbf\@charlb\p@pertitle\@charrb\@backslashchar\@backslashchar}
\makeatother
\closeout\contentsfile
}

\if@proofs\AddToShipoutPicture{\BackgroundPic}\fi

\usepackage{amsmath,amsthm,amssymb}
\usepackage{enumerate}
\usepackage{color}

\theoremstyle{plain}

\numberwithin{equation}{section}

\newcommand\cL{\mathcal L}

\newcommand\cC{\mathcal C}

\newcommand\cD{\mathcal D}

\newcommand\cH{\mathcal H}

\newcommand\cM{\mathcal M}

\newcommand\cX{\mathcal X}

\newcommand\cT{\mathcal T}

\newcommand\cS{\mathcal S}
\newcommand\cW{\mathcal W}

\newcommand\PG{{\rm PG}}

\newcommand\GF{{\rm GF}}

\theoremstyle{definition}

\begin{document}




\begin{frontmatter}   

\titledata{On Hermitian varieties in $\PG(6,q^2)$}           
{}                 

\authordata{Angela Aguglia}            
{Dipartimento di Meccanica, Matematica e Management, Politecnico di Bari, Via Orabona 4, I-70125 Bari, Italy}    
{angela.aguglia@poliba.it}                     
{The author was supported by  the Italian National Group for Algebraic and Geometric Structures and their Applications (GNSAGA - INdAM). }

\authordata{Luca Giuzzi}            
{DICATAM, University of Brescia, Via Branze 53,
  I-25123 Brescia, Italy}
{luca.giuzzi@unibs.it}
{The author was supported by  the Italian National Group for Algebraic and Geometric Structures and their Applications (GNSAGA - INdAM). }

\authordata{Masaaki Homma}            
{Department of Mathematics and Physics, Kanagawa University, Hiratsuka 259-1293, Japan}
{homma@kanagawa-u.ac.jp}
{}

\msc{51E21; 51E15; 51E20}

\begin{abstract}
 In this paper we  characterize the  non-singular Hermitian variety $\cH(6,q^2)$ of $\PG(6,q^2)$, $q\neq 2$ among the irreducible hypersurfaces of degree $q+1$ in $\PG(6,q^2)$ not containing solids
  by the number of its  points and the existence of
  a solid $S$ meeting it in $q^4+q^2+1$ points.
\end{abstract}

\end{frontmatter}

\section{Introduction}
The set of all absolute points of a non-degenerate unitary polarity in
$\PG(r,q^2)$ determines the Hermitian variety $\cH(r,q^2)$.
This is a non-singular
algebraic hypersurface of degree $q+1$ in $\PG(r,q^2)$ with a number of remarkable properties, both from
the geometrical and the combinatorial point of view; see~\cite{BC,SE2}.
In particular, $\cH(r,q^2)$ is a
$2$-character set with respect to the hyperplanes of $\PG(r,q^2)$ and $3$-character blocking set with respect to the lines of $\PG(r,q^2)$
for $r>2$.
An interesting and widely investigated problem is to provide combinatorial descriptions of $\cH(r,q^2)$ among all hypersurfaces of the same degree.

First, we observe that
a condition on the number of points and the intersection numbers
with hyperplanes is not in general sufficient to characterize Hermitian varieties; see~\cite{ACK12},\cite{A13}. On the other hand, it is
enough  to consider in addition the intersection numbers with codimension $2$ subspaces in order to get a complete description; see~\cite{DS1}.

In general, a hypersurface $\cH$ of $\PG(r,q)$ is viewed
as a hypersurface over the algebraic closure of $\GF(q)$ and a point of $\PG(r, q^i)$ in $\cH$ is called a
$\GF(q^i)$-point. A $\GF(q)$-point of $\cH$ is also said  to be a rational point of $\cH$. Throughout this paper, the
number of $\GF(q^i)$-points of $\cH$ will be denoted by $N_{q^i}(\cH)$. For simplicity, we shall also use the convention
$|\cH|=N_q(\cH)$.

In the present paper, we shall investigate a combinatorial characterization
of the Hermitian hypersurface $\cH(6,q^2)$ in $\PG(6,q^2)$ among all hypersurfaces of the same degree having also the same number of $\GF(q^2)$-rational points.

More in detail, in \cite{HK,HKf} it has been proved that if $\cX$ is a hypersurface of degree $q+1$ in $\PG(r,q^2)$, $r\geq3$ odd, with    $|\cX|=|\cH(r,q^2)|=(q^{r+1}+(-1)^r)(q^{r}-(-1)^r)/(q^2-1)$
$\GF(q^2)$-rational points,
not containing linear subspaces of dimension greater than
$\frac{r-1}{2}$, then $\cX$ is a non-singular Hermitian variety of $\PG(r,q^2)$.
This result generalizes the characterization of
\cite{HSTV} for the Hermitian curve of $\PG(2,q^2)$, $q\neq 2$.

The case where $r>4$ is even is, in general, currently open.
A starting point for a characterization in arbitrary even dimension can be found in~\cite{AAFP} where the case of a hypersurface $\cX$ of degree $q+1$
in $\PG(4,q^2)$, $q>3$ is considered. There,
it is shown that when $\cX$  has the same number of  rational points as $\cH(4,q^2)$, does not contain any subspaces of dimension greater than $1$ and meets at least
one plane $\pi$ in $q^2+1$  $\GF(q^2)$-rational points, then $\cX$ is a Hermitian variety.

In this article we deal with hypersurfaces of degree $q+1$ in $\PG(6,q^2)$
and we prove  that a characterization similar to that of~\cite{AAFP}
holds also in dimension $6$. We conjecture that this can be extended
to arbitrary even dimension.

\begin{thm}\label{main}
Let $S$ be a hypersurface of $\PG(6,q^2)$, $q>2$, defined over $\GF(q^2)$, not containing solids. If the degree of $S$ is $q+1$ and the number of its rational points is $q^{11}+q^9+q^7+q^4+q^2+1$, then every solid of $\PG(6,q^2)$ meets $S$ in at least $q^4+q^2+1$  rational points. If there is at least a solid $\Sigma_3$ such that $|\Sigma_3\cap S|=q^4+q^2+1$, then $S$ is a non-singular Hermitian variety of $\PG(6,q^2)$.
\end{thm}

Furthermore, we also extend the result obtained in \cite{AAFP} to the case $q=3$.

\section{Preliminaries and notation}

In this section we collect some useful information and results that will be crucial to our proof.

A Hermitian variety in $\PG(r,q^2)$ is the algebraic variety of
$\PG(r,q^2)$ whose points $\langle v\rangle$
satisfy the equation $\eta(v,v)=0$ where $\eta$ is a sesquilinear
form $\GF(q^2)^{r+1}\times\GF(q^2)^{r+1}\to\GF(q^2)$.
The \emph{radical} of the form $\eta$ is the vector subspace of
$\GF(q^2)^{r+1}$ given by
\[ \mathrm{Rad}(\eta):=\{ w\in\GF(q^2)^{r+1}\colon
  \forall v\in\GF(q^2)^{r+1}, \eta(v,w)=0 \}. \]
The form $\eta$ is non-degenerate if $\mathrm{Rad}(\eta)=\{\mathbf{0}\}$.
If the form $\eta$ is non-degenerate, then the corresponding
Hermitian variety is denoted by $\cH(r,q^2)$ and it is non-singular,
of degree $q+1$ and contains
$$(q^{r+1}+(-1)^r)(q^{r}-(-1)^r)/(q^2-1)$$
$\GF(q^2)$-rational points.
When $\eta$ is
degenerate we shall call \emph{vertex} $R_t$ of the degenerate
Hermitian variety associated to $\eta$ the projective
subspace $R_t:=\PG(\mathrm{Rad}(\eta)):=\{ \langle w\rangle\colon
w\in\mathrm{Rad}(\eta)\}$ of $\PG(r,q^2)$.
A degenerate Hermitian variety can always be described as
a cone of vertex $R_t$ and  basis a non-degenerate
Hermitian variety $\cH(r-t,q^2)$ disjoint from $R_t$ where $t=\dim(\mathrm{Rad}(\eta))$
is the vector dimension of the radical of $\eta$. In this case
we shall write the corresponding variety as $R_t\cH(r-t,q^2)$.
Indeed,
\[ R_t\cH(r-t,q^2):=\{ X\in\langle P,Q\rangle\colon
  P\in R_t, Q\in\cH(r-t,q^2) \}. \]

Any line of $\PG(r,q^2)$ meets a Hermitian variety (either degenerate or not)
in either $1, q+1$ or $q^2+1$ points (the latter value only for $r>2$).
The maximal dimension of projective subspaces contained in
the non-degenerate Hermitian variety $\cH(r,q^2)$ is $(r-2)/2$, if $r$ is even, or $(r-1)/2$, if $r$ is odd.
These subspaces of
maximal dimension are called {\em generators} of $\cH(r,q^2)$ and the generators of $\cH(r,q^2)$ through a point $P$ of $\cH(r, q^2)$ span a hyperplane
$P^{\perp}$ of $\PG(r,q^2)$, the {\em tangent hyperplane} at $P$.

It is well known that
this hyperplane meets $\cH(r,q^2)$ in a degenerate Hermitian variety
$P\cH(r-2,q^2)$, that is in a Hermitian
cone having as vertex the point $P$ and as base a non-singular Hermitian variety of $\Theta\cong\PG(r-2,q^2)$ contained in $P^{\perp}$ with
$P\not\in\Theta$.

Every hyperplane of $\PG(r, q^2)$, which is not tangent, meets $\cH(r, q^2)$ in a non-singular Hermitian variety $\cH(r-1, q^2)$, and is called a {\em secant hyperplane} of $\cH(r, q^2)$.
In particular, a tangent hyperplane contains
$$
1+q^2(q^{r-1}+(-1)^r)(q^{r-2}-(-1)^r)/(q^2-1)
$$
$\GF(q^2)$-rational points of $\cH(r,q^2)$, whereas a secant hyperplane contains
$$
(q^{r}+(-1)^{r-1})(q^{r-1}-(-1)^{r-1})/(q^2-1)
$$
$\GF(q^2)$-rational points of $\cH(r,q^2)$.

We now recall several results which shall be used in the course of
this paper.

\begin{lem}[\cite{SE}]\label{Se} Let $d$ be an integer with $1\leq d\leq q+1$ and let $\cC$ be a curve of degree $d$ in $\PG(2,q)$ defined over $\GF(q)$, which may have $\GF(q)$-linear components. Then the number  of its rational points is at most $dq+1$ and $N_{q}(\cC)=dq+1$ if and only if $\cC$ is a pencil of $d$ lines of $\PG(2,q)$.
\end{lem}

\begin{lem}[\cite{HK3}]\label{Szi}
Let $d$ be an integer with $2\leq d  \leq q+2$, and $\cC$
a curve of degree $d$ in $\PG(2,q)$ defined over $\GF(q)$ without any $\GF(q)$-linear components. Then $N_{q}(\cC)\leq (d-1)q+1$, except for a class of plane curves of degree $4$ over $\GF(4)$ having $14$ rational points.
\end{lem}

\begin{lem}[\cite{HK0}]\label{hk0}
Let $\cS$ be a surface of degree $d$ in $\PG(3,q)$ over $\GF(q)$. Then
\[N_{q}(\cS)\leq dq^2+q+1\]
\end{lem}




\begin{lem}[\cite{HSTV}]\label{uni}
Suppose $q\neq 2$.  Let $\cC$ be a plane curve over $\GF(q^2)$ of degree $q+1$ without $\GF(q^2)$-linear components. If $\cC$ has $q^3+1$ rational points, then $\cC$ is a Hermitian curve.
\end{lem}

\begin{lem}[\cite{DS1}]\label{ds1}
A subset of points of $\PG(r,q^2)$ having the same intersection numbers with respect to hyperplanes and spaces of codimension $2$ as non-singular Hermitian varieties, is a non-singular Hermitian variety of $\PG(r,q^2)$.
\end{lem}

From \cite[Th 23.5.1,Th 23.5.3]{HT} we have the following.
\begin{lem}\label{ch}
  If $\cW$ is a set of $q^7+q^4+q^2+1$ points of $\PG(4,q^2)$, $q > 2$, such that every line of $\PG(4,q^2)$ meets $\cW$ in $1, q+1$ or $q^2+1$ points, then $\cW$ is a
  Hermitian cone with vertex a line and base a unital.
\end{lem}
Finally, we recall that a {\em blocking set with respect to lines} of $\PG(r, q)$ is a point set which blocks all the lines, i.e., intersects each line of $\PG(r, q)$ in at least one point.

\section{Proof of Theorem~\ref{main}}
We first provide an estimate on the number of points of
a curve of degree $q+1$ in $\PG(2,q^2)$, where $q$ is any prime power.
\begin{lem}\label{Homma}
  Let $\cC$ be a plane curve over $\GF(q^2)$, without $\GF(q^2)$-lines as components and of degree $q+1$. If the number of $\GF(q^2)$-rational points
  of $\cC$ is $N<q^3+1$, then
\begin{equation} \label{num}
N\leq \left\{\begin{array}{ll}
    q^3-(q^2-2)  & {\mbox{\it if}  \ q>3}  \\
    24   & {\mbox{\it if}  \ q=3} \\
     8    & {\mbox{\it if}  \ q=2} .
   \end{array}  \right.
\end{equation}

\end{lem}
\begin{proof}
We distinguish the following three cases:
\begin{enumerate}[(a)]
\item\label{(a)} $\cC$ has two or more $\GF(q^2)$-components;
\item\label{(b)} $\cC$ is irreducible over $\GF(q^2)$, but not absolutely irreducible;
\item\label{(c)} $\cC$ is absolutely irreducible.
\end{enumerate}
Suppose first $q\neq2$.\par\noindent
{\bf{Case (\ref{(a)})}}
Suppose $\cC=\cC_1\cup\cC_2$. Let $d_i$ be the degree of $\cC_i$,  for each $i=1,2$. Hence $d_1+d_2=q+1$. By Lemma~\ref{Szi},
\[N \leq N_{q^2}(\cC_1)+N_{q^2}(\cC_2)\leq [(q+1)-2]q^2+2=q^3-(q^2-2)\]
{\bf{Case (\ref{(b)})}} Let $\cC'$ be an irreducible component of $\cC$ over the algebraic closure of $\GF(q^2)$. Let $\GF(q^{2t})$
be the minimum defining field of $\cC'$ and $\sigma$ be the Frobenius morphism of $\GF(q^{2t})$ over $\GF(q^2)$. Then
\[\cC=\cC' \cup \cC'^{\sigma} \cup \cC'^{\sigma^2}\cup \ldots \cup \cC'^{\sigma^{t-1}},    \]
and the degree of $\cC'$, say $e$, satisfies $q+1=te$ with $e>1$. Hence any $\GF(q^2)$-rational point of $\cC$ is contained in $\cap_{i=0}^{t-1} \cC'^{\sigma^i}$. In particular, $N\leq e^2\leq (\frac{q+1}{2})^2$ by Bezout's Theorem and $(\frac{q+1}{2})^2 <q^3-(q^2-2)$.\\
\noindent
{\bf{Case (\ref{(c)})}} Let $\cC$ be an absolutely irreducible curve over $\GF(q^2)$ of degree $q+1$. Either $\cC$ has a singular point or not.

In general, an absolutely irreducible plane curve $\cM$  over $\GF(q^2)$ is $q^2$-Frobenius non-classical if for a general point $P(x_0,x_1,x_2)$ of $\cM$ the point  $P^{q^2}=$ $P^{q^2}(x_0^{q^2} ,x_1^{q^2},x_2^{q^2})$ is on the tangent line to $\cM$ at the point $P$. Otherwise, the curve  $\cM$ is said to be Frobenius classical.
A lower bound of the number of $\GF(q^2)$-points for $q^2$-Frobenius non-classical curves is given by \cite[Corollary 1.4]{BH}:
for a $q^2$-Frobenius non-classical curve $\cC'$ of degree $d$,
we have $N_{q^2}(\cC') \geq d(q^2 -d +2)$.
In particular, if $d=q+1$, the lower bound is just $q^3 +1$.

Going back to our original curve $\cC$,
we know $\cC$ is Frobenius classical because $N < q^3 +1$.
Let $F(x,y,z)=0$ be an equation of $\cC$ over $\GF(q^2)$.
We consider the curve $\cD$ defined by
$
\frac{\partial F}{\partial x}x^{q^2}
+
\frac{\partial F}{\partial y}y^{q^2}
+
\frac{\partial F}{\partial z}z^{q^2} =0.
$
Then $\cC$ is not a component of $\cD$ because $\cC$ is Frobenius classical.
Furthermore, any $\GF(q^2)$-point $P$ lies on $\cC \cap \cD$
and the intersection multiplicity of $\cC$ and $\cD$ at $P$ is at least $2$
by Euler's theorem for homogeneous polynomials.
Hence by B\'{e}zout's theorem,
$
2N \leq (q+1)(q^2 +q).
$
Hence
\[
N \leq \frac{1}{2}q(q+1)^2.
\]
This argument is due to St\"{o}hr and Voloch \cite[Theorem 1.1]{SV}.
This St{\"o}hr and  Voloch's bound is lower than the estimate for $N$ in case~(\ref{(a)}) for $q>4$ and it is the same for $q=4$.
When
$q=3$ the bound in case (\ref{(a)}) is smaller than the
St{\"o}hr and  Voloch's bound.\\
\noindent
Finally, we consider the case $q=2$. Under this assumption,
$\cC$ is a cubic curve and neither case (\ref{(a)}) nor case (\ref{(b)}) might occur. For a degree $3$ curve over $\GF(q^2)$ the St{\"o}hr and  Voloch's bound is loose, thus we need to change our argument.
If $\cC$ has a singular point, then $\cC$ is a rational curve with a unique singular point. Since the degree of $\cC$ is $3$, singular points are either cusps or ordinary double points. Hence $N\in\{4,5,6\}$. If $\cC$ is nonsingular, then it is an elliptic curve and, by
the Hasse-Weil bound, see \cite{AW},  $N\in I$ where $I=\{1,2,\ldots, 9\}$ and for each number $N$ belonging to  $I$ there is an elliptic curve over $\GF(4)$ with $N$ points, from \cite[Theorem 4.2]{SH}.
This completes the proof.
\end{proof}
Henceforth, we shall always suppose $q>2$ and
we denote by $\cS$  an algebraic hypersurface of $\PG(6,q^2)$ satisfying the
following hypotheses of Theorem~\ref{main}:
\begin{enumerate}[(S1)]
\item\label{S1}
  $\cS$ is an algebraic hypersurface of degree $q+1$ defined over $\GF(q^2)$;
\item\label{S2}
  $|\cS|=q^{11}+q^9+q^7+q^4+q^2+1$;
\item\label{S3} $\cS$ does not contain projective $3$-spaces (solids);
\item\label{S4} there exists a solid $\Sigma_3$ such that
  $|\cS\cap\Sigma_3|=q^4+q^2+1$.
\end{enumerate}
We first consider the behavior of $\cS$ with respect to the lines.
\begin{lem}
  \label{block}
  An algebraic hypersurface $\cT$ of degree $q+1$
  in $\PG(r,q^2)$, $q\neq 2$, with $|\cT|=|\cH(r,q^2)|$
  is a blocking set with respect to lines of $\PG(r,q^2)$
\end{lem}
\begin{proof}
Suppose on the contrary that there is a line $\ell$
of $\PG(r,q^2)$ which is disjoint from $\cT$.
Let $\alpha$ be a plane containing  $\ell$.
The algebraic  plane curve $\cC=\alpha \cap\cT$ of degree $q+1$ cannot have $\GF(q^2)$-linear components and hence it has at most $q^3+1$ points
because of Lemma~\ref{Szi}.
If $\cC$ had $q^3+1$  rational points, then from Lemma~\ref{uni},
$\cC$ would be a Hermitian curve with an external line, a contradiction since Hermitian curves are blocking sets.
Thus $N_{q^2}(\cC)\leq q^3$.
Since $q>2$, by Lemma~\ref{Homma}, $N_{q^2}(\cC)<q^3-1$  and hence every plane through $r$ meets $\cT$ in at most $q^3-1$ rational points.
Consequently, by considering all planes through $r$, we can bound the number of rational points of $\cT$ by
$N_{q^2}(\cT)\leq (q^3-1)\frac{q^{2r-4}-1}{q^2-1}=q^{2r-3}+\dots<|\cH(r,q^2)|$,
which is a contradiction.
Therefore there are no external lines to $\cT$ and so  $\cT$ is a blocking set w.r.t. lines of $\PG(r,q^2)$.
\end{proof}
\begin{remark}
  The proof of~\cite[Lemma 3.1]{AAFP} would work
  perfectly well here under the assumption $q>3$.
  The alternative argument of Lemma~\ref{block}
  is simpler and also holds for $q=3$.
\end{remark}

By the previous Lemma and assumptions~(S\ref{S1}) and
(S\ref{S2}), $\cS$ is a blocking set for the lines of $\PG(6,q^2)$
In particular,
the intersection of $\cS$ with any $3$-dimensional subspace $\Sigma$ of $\PG(6,q^2)$ is
also a blocking set with respect to lines of $\Sigma$
and hence it contains at least $q^4+q^2+1$ $\GF(q^2)$-rational points; see~\cite{BB}.

\begin{lem}
  \label{plane}
  Let $\Sigma_3$ be a solid of $\PG(6,q^2)$
  satisfying condition~(S\ref{S4}), that is
  $\Sigma_3$ meets
  $\cS$ in exactly $q^4+q^2+1$ points. Then,
  $\Pi:=\cS\cap\Sigma_3$ is a plane.
\end{lem}
\begin{proof}
  $\cS\cap\Sigma_3$ must be a blocking set for the lines of $\PG(3,q^2)$;
  also it has
  size $q^4+q^2+1$. It follows from~\cite{BB} that $\Pi:=\cS\cap\Sigma_3$ is
  a plane.
\end{proof}
\begin{lem}
  \label{cone0}
  Let $\Sigma_3$ be a solid of satisfying condition~(S\ref{S4}).
  Then,   any $4$-dimensional projective space $\Sigma_4$ through $\Sigma_3$
  meets $\cS$ in a Hermitian cone with vertex a line
  and basis a Hermitian curve.
\end{lem}

\begin{proof}
  Consider all of the $q^6+q^4+q^2+1$
  subspaces $\overline{\Sigma}_3$ of dimension $3$  in $\PG(6,q^2)$
   containing $\Pi$.

   From Lemma~\ref{hk0} and condition~(S\ref{S3}) we have
   $|\overline{\Sigma}_3 \cap  \cS|\leq q^5+q^4+q^2+1$. Hence,
   \[|\cS|=(q^7+1)(q^4+q^2+1)\leq (q^6+q^4+q^2)q^5+q^4+q^2+1=|\cS|\]
   Consequently, $|\overline{\Sigma}_3\cap\cS|= q^5+q^4+q^2+1$ for all  $\overline{\Sigma}_3 \neq \Sigma_3$ such that $\Pi \subset \overline{\Sigma}_3$.

   Let $C:=\Sigma_4\cap\cS$. Counting the number of rational points of $C$ by considering the intersections with  the $q^2+1$ subspaces $\Sigma'_3$ of
   dimension $3$ in $\Sigma_4$ containing the plane $\Pi$  we get
  \[ |C|=q^2\cdot q^5+q^4+q^2+1=q^7+q^4+q^2+1. \]
  In particular, $C\cap\Sigma'_3$ is
  a maximal surface of degree $q+1$; so it must split in $q+1$ distinct
  planes through a line of $\Pi$; see~\cite{S92}.
  So $C$ consists of $q^3+1$ distinct planes belonging to distinct $q^2$ pencils, all  containing $\Pi$ ; denote by ${\cL}$
  the family of these planes. Also for
  each $\Sigma_3'\neq\Sigma_3$, there is a line $\ell'$
  such that all the planes
  of $\mathcal L$ in $\Sigma_3'$ pass through $\ell'$.
  It is now straightforward to see that any line contained in $C$
  must necessarily belong to one of the planes of $\mathcal L$ and
  no plane not in $\mathcal L$ is contained in $C$.

  In order to get the result it is now enough to show that
  a line of $\Sigma_4$ meets $C$ in either $1$, $q+1$ or $q^2+1$
  points. To this purpose, let $\ell$ be a line of $\Sigma_4$ and
  suppose $\ell\not\subseteq C$. Then, by Bezout's theorem,
  \[ 1\leq |\ell\cap C|\leq q+1. \]
  Assume $|\ell\cap C|>1$. Then we can distinguish two cases:
  \begin{enumerate}
  \item $\ell\cap\Pi\neq\emptyset$. If $\ell$ and $\Pi$ are incident,
    then we can consider the $3$-dimensional subspace
    $\Sigma_3':=\langle\ell,\Pi\rangle$. Then $\ell$ must meet
    each plane of $\mathcal L$ in $\Sigma_3'$ in different points (otherwise
    $\ell$ passes through the intersection of these planes and
    then $|\ell\cap C|=1$). As there are $q+1$ planes of $\mathcal L$
    in $\Sigma_3'$, we have $|\ell\cap C|=q+1$.
  \item $\ell\cap\Pi=\emptyset$. Consider the plane $\Lambda$
    generated by a point $P\in\Pi$ and $\ell$.
    Clearly $\Lambda\not\in{\mathcal L}$.
    The curve
    $\Lambda\cap S$ has degree $q+1$ by construction, does not
    contain lines (for otherwise $\Lambda\in{\mathcal L}$) and
    has $q^3+1$ $\GF(q^2)$-rational
    points (by a counting argument). So from Lemma \ref{uni} it is
    a Hermitian curve . It follows that $\ell$ is a $q+1$ secant.
  \end{enumerate}
  We can now apply   Lemma \ref{ch} to see
  that $C_1$ is a Hermitian cone with vertex a line.
\end{proof}
\begin{lem}
  \label{cone1}
  Let $\Sigma_3$ be a space satisfying condition~(S\ref{S4}) and take
  $\Sigma_5$ to be a $5$-dimensional projective space with
  $\Sigma_3\subseteq\Sigma_5$.
  Then $\cS\cap\Sigma_5$ is a Hermitian cone with vertex a point
  and basis a Hermitian hypersurface $\cH(4,q^2)$.
\end{lem}
\begin{proof}
  Let \[ \Sigma_4:=\Sigma_4^1, \Sigma_4^2,\dots, \Sigma_4^{q^2+1} \]
  be the $4$-spaces through $\Sigma_3$ contained in $\Sigma_5$.
  Put $C_i:=\Sigma_4^i\cap\cS$, for all  $i \in \{1,\ldots, q^2+1\}$ and $\Pi=\Sigma_3\cap\cC_1$.
  From Lemma \ref{cone0} $C_i$ is a Hermitian cone with vertex a line, say $\ell_i$. Furthermore
    $\Pi\subseteq\Sigma_3\subseteq\Sigma_4^i$ where $\Pi$ is a plane.
  Choose a plane $\Pi'\subseteq\Sigma_4^1$ such that
  $m:=\Pi'\cap C_1$ is a line $m$ incident with $\Pi$ but not contained
  in it. Let $P_1:=m\cap\Pi$.
  It is straightforward to see that in $\Sigma_4^1$
  there is exactly $1$ plane
  through $m$ which is a $(q^4+q^2+1)$-secant, $q^4$ planes which are
  $(q^3+q^2+1)$-secant and $q^2$ planes which are $(q^2+1)$-secant.
  Also $P_1$ belongs to the line $\ell_1$.
  There are now two cases to consider:
  \begin{enumerate}[(a)]
  \item\label{(a2)} There is a plane $\Pi''\neq\Pi'$ not contained
    in $\Sigma_4^i$ for all $i=1,\dots,q^2+1$ with
    $m\subseteq\Pi''\subseteq S\cap\Sigma_5$.

    We first show that the vertices of the cones
    $C_i$ are all concurrent.
    Consider $m_i:=\Pi''\cap\Sigma_4^i$. Then
    $\{m_i: i=1,\dots,q^2+1\}$ consists of $q^2+1$ lines
    (including $m$) all through $P_1$.
    Observe that for all $i$, the line $m_i$ meets the
    vertex $\ell_i$ of the cone $C_i$ in $P_i\in\Pi$.
    This forces $P_1=P_2=\dots=P_{q^2+1}$. So
    $P_1\in\ell_1,\dots,\ell_{q^2+1}$.

    Now let $\overline{\Sigma}_4$ be a $4$-dimensional space in
    $\Sigma_5$ with
    $P_1\not\in\overline{\Sigma}_4$; in particular
    $\Pi\not\subseteq\overline{\Sigma}_4$.
    Put also $\overline{\Sigma}_3:=\Sigma_4^1\cap\overline{\Sigma}_4$.
    Clearly, $r:=\overline{\Sigma}_3\cap\Pi$ is a line and
    $P_1\not\in r$. So $\overline{\Sigma}_3\cap\cS$ cannot be
    the union of $q+1$ planes, since if this were to be the
    case, these planes would have to pass through the vertex
    $\ell_1$. It follows that $\overline{\Sigma}_3\cap\cS$ must
    be a Hermitian cone with vertex a point and basis a
    Hermitian curve. Let
    $\cW:=\overline{\Sigma}_4\cap\cS$.
    The intersection $\cW\cap\Sigma_4^i$ as $i$ varies
    is a Hermitian cone with basis a Hermitian curve, so,
    the points of $\cW$ are
    \[ |\cW|=(q^2+1)q^5+q^2+1=(q^2+1)(q^5+1); \]
    in particular, $\cW$ is a hypersurface of $\overline{\Sigma}_4$ of
    degree $q+1$ such that there
    exists a plane of $\overline{\Sigma}_4$ meeting
    $\cW$ in just one line (such planes exist in
    $\overline{\Sigma}_3$).
    Also suppose $\cW$ to contain planes and let $\Pi'''\subseteq\cW$
    be such a plane.
    Since  $\Sigma_4^i\cap\cW$ does not contain planes,
    all  $\Sigma_4^i$ meet $\Pi'''$ in a line $t_i$. Also
    $\Pi'''$ must be contained in $\bigcup_{i=1}^{q^2+1} t_i$.
    This implies that the set $\{t_i\}_{i=1,\dots,q^2+1}$ consists
    of $q^2+1$ lines through a point $P\in \Pi \setminus \{P_1\}$.

    Furthermore each line $t_i$ passing through $P$  must meet the radical line $\ell_i$ of the Hermitian cone $\cS \cap\Sigma_4^i$  and this forces $P$ to coincide with $P_1$, a contradiction.
    It follows that $\cW$ does not contain planes.

        So by the characterization of $\cH(4,q^2)$
    of~\cite{AAFP}
    we have that $\cW$ is a Hermitian variety $\cH(4,q^2)$.

    We also have that $|\cS\cap\Sigma_5|=|P_1\cH(4,q^2)|$. Let now $r$ be any
    line of $\cH(4,q^2)=\cS\cap\overline{\Sigma}_4$ and let $\Theta$ be
    the plane $\langle r,P_1\rangle$. The plane $\Theta$ meets
    $\Sigma_4^i$ in a line $q_i\subseteq\cS$ for each $i=1,\dots,q^2+1$ and these lines
    are concurrent in $P_1$. It follows that all the points of $\Theta$ are
    in $S$. This completes the proof for the current case and shows that
    $\cS\cap\Sigma_5$ is a Hermitian cone $P_1\cH(4,q^2)$.
  \item\label{(b2)} All planes $\Pi''$ with
    $m\subseteq\Pi''\subseteq\cS\cap\Sigma_5$
    are contained in $\Sigma_4^i$ for some $i=1,\dots,q^2+1$.
    We claim that this case cannot happen.
    We can suppose without loss of generality $m\cap\ell_1=P_1$
    and $P_1\not\in\ell_i$ for all $i=2,\dots,q^2+1$.
    Since the intersection of the  subspaces $\Sigma_4^i$  is $\Sigma_3$,  there is exactly one plane through $m$  in $\Sigma_5$ which is
    $(q^4+q^2+1)$-secant,
    namely the plane $\langle\ell_1,m\rangle$.
    Furthermore, in $\Sigma_4^1$ there are $q^4$ planes
    through $m$ which are $(q^3+q^2+1)$-secant and
    $q^2$ planes which are $(q^2+1)$-secant.
    We can provide an upper bound to the points of $\cS\cap\Sigma_5$ by
    counting the number of points of $\cS\cap\Sigma_5$ on  planes in $\Sigma_5$
    through $m$ and observing that a plane through $m$ not
    in $\Sigma_5$ and not contained in $\cS$
    has at most $q^3+q^2+1$ points in common with $\cS\cap\Sigma_5$.
    So
    \[ |\cS\cap\Sigma_5|\leq q^6\cdot q^3+q^7+q^4+q^2+1. \]
    As
    $|\cS\cap\Sigma_5|=q^9+q^7+q^4+q^2+1$,
    all planes through $m$ which are neither $(q^4+q^2+1)$-secant
    nor $(q^2+1)$-secant are $(q^3+q^2+1)$-secant. That is to say
    that all of these planes meet $\cS$ in a curve of degree $q+1$
    which must split into $q+1$ lines through a point because of Lemma \ref{Se}.

    Take now $P_2\in\Sigma_4^2\cap\cS$ and consider the plane
    $\Xi:=\langle m,P_2\rangle$. The line $\langle P_1,P_2\rangle$
    is contained in $\Sigma_4^2$; so it must be a $(q+1)$-secant,
    as it does not meet the vertex line $\ell_2$ of $C_2$ in
    $\Sigma_4^2$. Now, $\Xi$ meets every of $\Sigma_4^i$ for
    $i=2,\dots,q^2+1$ in a line through $P_1$ which is either a
    $1$-secant or a $q+1$-secant; so
    \[ |\cS\cap\Xi|\leq q^2(q)+q^2+1=q^3+q^2+1. \]
    It follows $|\cS\cap\Xi|=q^3+q^2+1$ and $\cS\cap\Xi$ is
    a set of $q+1$ lines all through the point $P_1$.
    This contradicts our previous construction.
  \end{enumerate}
  \end{proof}

  \begin{lem}
    \label{sections}
       Every hyperplane of $\PG(6,q^2)$ meets $\cS$
    either in  a non-singular Hermitian variety
    $\cH(5,q^2)$ or in a cone over a Hermitian hypersurface $\cH(4,q^2)$.
  \end{lem}
  \begin{proof}
    Let $\Sigma_3$ be a solid satisfying condition~(S\ref{S4}).
    Denote by $\Lambda$  a hyperplane of $\PG(6,q^2)$. If $\Lambda$ contains $\Sigma_3$ then, from Lemma \ref{cone1} it follows that
    $\Lambda\cap\cS$ is a Hermitian cone $P\cH(4,q^2)$.

   Now assume that $\Lambda$ does not contain   $\Sigma_3$.
   Denote by $S_5^j$, with $j=1,\ldots,q^2+1$ the $q^2+1$ hyperplanes through $\Sigma_4^1$, where as before, $\Sigma_4^1$ is a $4$-space containing $\Sigma_3$. By Lemma~\ref{cone1} again we get that
   $S_5^j \cap\cS=P^j\cH(4,q^2)$.
  We count  the number of rational points of $\Lambda\cap\cS$ by
  studying the intersections  of  $S_5^j \cap\cS$ with $\Lambda$
  for all $j \in \{1,\ldots, q^2+1\}$.
  Setting  $\cW_j:=S_5^j \cap\cS \cap \Lambda$,
  $\Omega:=\Sigma_4^1 \cap\cS \cap \Lambda$ then

    \[|\cS\cap \Lambda|=\sum_{j}|\cW_j\setminus\Omega|+|\Omega|.\]
    If $\Pi$ is a plane of $\Lambda$
    then  $\Omega$ consists of $q+1$  planes of a pencil.
    Otherwise let $m$ be the line in which $\Lambda$
    meets the plane $\Pi$. Then $\Omega$
    is  either a Hermitian cone $P_0\cH(2,q^2)$, or $q+1$
    planes of a pencil, according as the vertex $P^j\in \Pi$
    is an external point with respect to $m$ or not.

    In the former case $\cW_j$ is a non singular Hermitian variety
    $\cH(4,q^2)$ and thus
    $ |\cS \cap \Lambda|= (q^2+1)(q^7) +q^5+q^2+1 = q^9+q^7+q^5+q^2+1$.

    In the case in which $\Omega$ consists of $q+1$  planes of a pencil then
    $\cW_j$  is  either a  $P_0\cH(3,q^2)$  or a Hermitian cone with vertex a line and basis a Hermitian curve $\cH(2,q^2)$.

   If there is at least one index $j$ such that $\cW_j= \ell_1\cH(2,q^2)$ then,  there must be a $3$-dimensional space $\Sigma'_3$   of $S_5^j \cap\Lambda$ meeting $\cS$ in a generator. Hence, from Lemma \ref{cone1} we get that $\cS \cap\Lambda$ is a Hermitian cone $ P'\cH(4,q^2)$.

Assume that for all $j \in \{1,\ldots, q^2+1\}$,  $\cW_j$  is  a  $P_0\cH(3,q^2)$. In this case
  \[ |\cS \cap\Lambda|=(q^2+1)q^7+ (q+1)q^4+q^2+1=q^9+q^7+q^5+q^4+q^2+1 =|\cH(5,q^2)|. \]

We are going to prove that  the intersection numbers of $\cS$ with hyperplanes are only two that is   $q^9+q^7+q^5+q^4+q^2+1$ or  $q^9+q^7+q^4+q^2+1$.

Denote by $x_i$ the number of hyperplanes meeting $\cS$ in $i$
    rational points with $i \in \{q^9+q^7+q^4+q^2+1, q^9+q^7+q^5+q^2+1, q^9+q^7+q^5+q^4+q^2+1\}$. Double counting arguments give the following equations for  the integers $x_i$:
\begin{equation} \label{tg01}
\left\{\begin{array}{l}
    \sum_{i} x_i=q^{12}+q^{10}+q^8+q^6+q^4+q^2+1 \\
    \\
    \sum_{i}ix_i=|\cS|(q^{10}+q^8+q^6+q^4+q^2+1)\\
    \\
     \sum_{i=1}i(i-1)x_i=|\cS| (|\cS|-1)(q^8+q^6+q^4+q^2+1).
   \end{array}  \right.
\end{equation}
Solving \eqref{tg01} we obtain $x_{q^9+q^7+q^5+q^2+1}=0$.
In the case in which  $|\cS \cap \Lambda|=|\cH(5,q^2)|$, since $\cS\cap\Lambda$ is an algebraic hypersurface of degree $q+1$ not containing $3$-spaces, from  \cite[Theorem 4.1]{AW} we get that  $\cS\cap\Lambda$ is a Hermitian variety $\cH(5,q^2)$ and  this completes the proof.
  \end{proof}
\begin{proof}[Proof of Theorem~\ref{main}]
The first part of  Theorem \ref{main} follows from Lemma ~\ref{plane}.
  From Lemma \ref{sections},  $\cS$ has the same intersection numbers with respect to hyperplanes and $4$-spaces as  a non-singular Hermitian variety of $\PG(6,q^2)$, hence  Lemma \ref{ds1}  applies and  $\cS$ turns out to be a $\cH(6,q^2)$.
\end{proof}

\begin{rem}
  The characterization of the non-singular Hermitian variety $\cH(4,q^2)$
  given in \cite{AAFP} is based on the property that a given hypersurface is a blocking set with respect to lines of $\PG(4,q^2)$, see \cite[Lemma 3.1]{AAFP}. This lemma holds when $q>3$.
Since Lemma~\ref{block} extends the same  property to the case $q=3$ it follows that the
result stated in \cite{AAFP} is also valid in $\PG(4,3^2)$.
\end{rem}

\section{Conjecture}

We propose a conjecture for the general $2n$-dimensional case.

\emph{
Let $\cS$ be a hypersurface of $\PG(2d,q^2)$, $q>2$, defined over $\GF(q^2)$, not containing $d$-dimensional projective subspaces. If the degree of $\cS$ is $q+1$ and the number of its rational points is $|\cH(2d,q^2)|$, then every $d$-dimensional subspace of $\PG(2d,q^2)$ meets $\cS$ in at least $\theta_{q^2}(d-1):=(q^{2d-2}-1)/(q^2-1)$  rational points. If there is at least a $d$-dimensional subspace $\Sigma_d$  such that $|\Sigma_d\cap \cS|=|\PG(d-1,q^2)|$, then $\cS$ is a non-singular Hermitian variety of $\PG(2d,q^2)$.
 }

 Lemma \ref{Homma} and Lemma \ref{block}  can be a starting point for the proof of this  conjecture  since from them we get that $\cS$ is a blocking set with respect to lines of $\PG(2d,q^2)$.


\end{document}
